\documentclass[11pt,a4paper]{article}
\usepackage{amsfonts,amsgen,amstext,amsbsy,amsopn,amsfonts,amssymb,amscd}
\usepackage[leqno]{amsmath}
\usepackage{dsfont}
\usepackage[amsmath,amsthm,thmmarks]{ntheorem}
\usepackage{listings}
\usepackage{epsf,epsfig}
\usepackage{float}
\usepackage{ebezier,eepic}
\usepackage{color}
\usepackage{tikz}
\usepackage{multirow}
\usepackage{mathrsfs}
\usepackage{graphicx}
\usepackage{subfigure}
\newtheorem{thm}{Theorem}[section]

\newtheorem{prop}[thm]{Proposition}

\newtheorem{lem}[thm]{Lemma}
\newtheorem{false statement}{False statement}
\newtheorem{cor}[thm]{Corollary}
\newtheorem{fact}[thm]{Fact}

\theoremstyle{definition}
\newtheorem{defn}[thm]{Definition}
\newtheorem{claim}[thm]{Claim}

\newtheorem{conj}[thm]{Conjecture}

\makeatletter \@addtoreset{equation}{section}

\def\hh{\mathcal{H}}

\def\hht{\mathcal{T}}
\def\hf{\mathcal{F}}
\def\hg{\mathcal{G}}

\def\ha{\mathcal{A}}
\def\hb{\mathcal{B}}

\def\hj{\mathcal{J}}
\def\hl{\mathcal{L}}
\def\hp{\mathcal{P}}
\def\hq{\mathcal{Q}}

\begin{document}

\title{\bf\Large  Improved bounds on the maximum diversity of intersecting families}
\date{}
\author{Peter Frankl$^1$, Jian Wang$^2$\\[10pt]
$^{1}$R\'{e}nyi Institute, Budapest, Hungary\\[6pt]
$^{2}$Department of Mathematics\\
Taiyuan University of Technology\\
Taiyuan 030024, P. R. China\\[6pt]
E-mail:  $^1$frankl.peter@renyi.hu, $^2$wangjian01@tyut.edu.cn
}
\maketitle

\begin{abstract}
 A family $\hf\subset \binom{[n]}{k}$ is called an intersecting family if $F\cap F'\neq \emptyset$ for all $F,F'\in \hf$.  If  $\cap \hf\neq \emptyset$ then $\hf$ is called a star. The diversity of an intersecting family $\hf$ is defined as the minimum number of $k$-sets in $\hf$, whose deletion results in a star. In the present paper, we prove that for $n>36k$ any intersecting family $\hf\subset \binom{[n]}{k}$ has diversity at most $\binom{n-3}{k-2}$, which improves the previous best bound $n>72k$ due to the first author. This result is derived from some strong bounds concerning the maximum degree of large intersecting families. Some related results are established as well.
\end{abstract}

\section{Introduction}

Let $[n]=\{1,\ldots,n\}$ be the standard $n$-element set, $\binom{[n]}{k}$ the collection of its $k$-subsets, $0\leq k\leq n$.  A subset $\hf\subset \binom{[n]}{k}$ is called a {\it $k$-uniform family} or simply {\it $k$-graph}. The family $\hf$ is said to be {\it intersecting} if $F\cap F'\neq\emptyset$ for all $F,F'\in \hf$. Similarly two families $\hf,\hg$ are called {\it cross-intersecting} if $F\cap G\neq \emptyset$ for all $F\in \hf$, $G\in \hg$. Set $\cap \hf =\mathop{\cap}\limits_{F\in \hf}F$. If $\cap \hf \neq \emptyset$ then $\hf$ is called a {\it star}. Stars are the simplest examples of intersecting families. The quintessential Erd\H{o}s-Ko-Rado Theorem shows that they are the largest as well.

\begin{thm}[\cite{ekr}]
Suppose that $n\geq 2k>0$, $\hf\subset\binom{[n]}{k}$ is intersecting, then
\begin{align}\label{ineq-ekr}
|\hf| \leq \binom{n-1}{k-1}.
\end{align}
\end{thm}

In the case $n=2k$, $\binom{n-1}{k-1}=\frac{1}{2}\binom{n}{k}$ and being intersecting is equivalent to $|\hf\cap\{F,[n]\setminus F\}|\leq 1$ for every $F\in \binom{[n]}{k}$. Consequently there are $2^{\binom{n-1}{k-1}}$ intersecting families $\hf\subset \binom{[2k]}{k}$ attaining equality in \eqref{ineq-ekr}. However for $n>2k\geq 4$ there is a strong stability.

\begin{thm}[Hilton-Milner Theorem \cite{HM67}]
Suppose that $n> 2k\geq 4$, $\hf\subset\binom{[n]}{k}$ is intersecting and $\hf$ is not a star, then
\begin{align}\label{ineq-hm}
|\hf| \leq \binom{n-1}{k-1}-\binom{n-k-1}{k-1}+1.
\end{align}
\end{thm}

Let us define the Hilton-Milner Family
\[
\hh(n,k) = \left\{F\in \binom{[n]}{k}\colon 1\in F,\ F\cap [2,k+1]\neq \emptyset\right\}\cup \{[2,k+1]\},
\]
showing that \eqref{ineq-hm} is best possible.

Let us define also the triangle family
\[
\hht(n,k) =\left\{T\in \binom{[n]}{k}\colon |T\cap [3]|\geq 2\right\}.
\]

 It is easy to check that $\hh(n,2)=\hht(n,2)$ and $|\hh(n,3)|=|\hht(n,3)|$ but they are not isomorphic.
 For $k\geq 4$, $|\hh(n,k)|>|\hht(n,k)|$. By now there are myriads of results proving and reproving, strengthening and generalizing these classical results \cite[etc.]{AK,Borg,F78-2,F2017,FFuredi,FT92,FW2022,GH,KZ,Mors}

Define
\[
\hf(i)=\left\{F\setminus \{i\}\colon i\in F\in \hf\right\},\ \hf(\bar{i})= \{F\colon i\notin F\in \hf\}
\]
and note that $|\hf|=|\hf(i)|+|\hf(\bar{i})|$. It should be pointed out that $\hf(1)\subset \binom{[2,n]}{k-1}$ and $\hf(\bar{1})\subset\binom{[2,n]}{k}$ are cross-intersecting.

There are several ways of measuring how far an intersecting family is from a star. Let us present two of them first. Define
\[
\varrho(\hf) =\max_{i\in [n]} \frac{|\hf(i)|}{|\hf|},\ \gamma(\hf) = \min_{i\in [n]} |\hf(\bar{i})|.
\]
If $\hf$ is a star then $\varrho(\hf) =1$ and $\gamma(\hf)=0$. That is, the larger $\varrho(\hf)$ is and the smaller $\gamma(\hf)$ is, the closer $\hf$ is to a star. In case of the Hilton-Milner Family $\gamma(\hf)=1$ and $\varrho(\hf)\rightarrow 1$ as $n \rightarrow \infty$.

Before proceeding further let us define a sequence of intersecting families bridging the triangle family $\hht(n,k)$ to the Hilton-Milner family $\hh(n,k)$.

For $2\leq r\leq k$ define
\[
\ha_r(n,k) =\left\{A\in \binom{[n]}{k}\colon 1\in A, A\cap [2,r+1]\neq \emptyset\right\} \bigcup \left\{A\in \binom{[n]}{k}\colon 1\notin A,[2,r+1]\subset A\right\}.
\]
Clearly, $\ha_2(n,k)=\hht(n,k)$, $\ha_k(n,k) = \hh(n,k)$ and $\gamma(\ha_r(n,k))=\binom{n-r-1}{k-r}$. Note also that 1 is the element of highest degree in $\ha_r(n,k)$ and this degree is
\[
\Delta(\ha_r(n,k)) =\binom{n-1}{k-1} -\binom{n-r-1}{k-1}.
\]

Let us note also that $|\ha_2(n,k)|=|\ha_3(n,k)|$. This can be shown by manipulating with binomial coefficients. However, it is simpler to consider the trace of both families on [4]:
\begin{align*}
\ha_2(n,k) &= \left\{A\in \binom{[n]}{k}\colon |A\cap [4]|\geq 3 \mbox{ \rm or } A\cap [4]\in \{(1,2),(1,3),(2,3)\}\right\},\\[5pt]
\ha_3(n,k) &= \left\{A\in \binom{[n]}{k}\colon |A\cap [4]|\geq 3 \mbox{ \rm or } A\cap [4]\in \{(1,2),(1,3),(1,4)\}\right\}.
\end{align*}

In any case, for fixed $k$ and $n\rightarrow \infty$,
\begin{align*}
\varrho(\ha_2(n,k)) = \frac{2}{3}+o(1) \mbox{ \rm while } \varrho(\ha_r(n,k))& = 1-o(1) \mbox{ \rm for } 3\leq r\leq k.
\end{align*}

In \cite{F78-2} the sunflower method was introduced and it was proved that $|\hf|>|\ha_2(n,k)|$ implies $\varrho(\hf) =1-o(1)$. The sunflower method has proved very powerful in solving several related and unrelated problems in extremal set theory (cf. e.g. \cite{Furedi,FFuredi2}) however it only works for $n>n_0(k)$, even in the best case for $n>ck^2$ (cf. \cite{FW2022-2}).

Another powerful method is the use of the Kruskal-Katona Theorem (\cite{Kruskal,Katona66}), especially its reformulation due to Hilton \cite{Hilton}. To state it recall the {\it lexicographic order} $A <_{L} B$ for $A,B\in \binom{[n]}{k}$ defined by, $A<_L B$ iff $\min\{i\colon i\in A\setminus B\}<\min\{i\colon i\in B\setminus A\}$. E.g., $(1,2,9)<_L(1,3,4)$.

For $n>k>0$ and $\binom{n}{k}\geq m>0$ let $\hl(n,k,m)$ denote the first $m$ sets $A\in \binom{[n]}{k}$ in the lexicographic order.

{\noindent\bf Hilton's Lemma (\cite{Hilton}).} Let $n,a,b$ be positive integers, $n\geq a+b$. Suppose that $\ha\subset \binom{[n]}{a}$ and $\hb\subset \binom{[n]}{b}$ are cross-intersecting. Then $\hl(n,a,|\ha|)$ and $\hl(n,b,|\hb|)$ are cross-intersecting as well.

This result played a crucial role in the proof of the following.

Define $\Delta(\hf)=\max\{|\hf(i)|\colon 1\leq i\leq n\}$, the {\it maximum degree}.

\begin{thm}[\cite{F87-2}]\label{thm-f87}
Let $n>2k\geq 4$ be integers. Suppose that $\hf\subset\binom{[n]}{k}$ is intersecting, $\Delta(\hf)\leq \Delta(\ha_r(n,k))$ for some $2\leq r\leq k$. Then
\begin{align}\label{ineq-f87}
|\hf|\leq |\ha_r(n,k)|.
\end{align}
\end{thm}

At the time of the writing of the above paper the notion of diversity did not exist yet (it is due to Lemmons and Palmer \cite{LP}). However, the actual proof yields \eqref{ineq-f87} under the assumption $\gamma(\hf)\geq \gamma(\ha_r(n,k))=\binom{n-r-1}{k-r}$. Let us mention that Kupavskii and Zakharov \cite{KZ} extended this result by proving strong upper bounds on $|\hf|$ under similar assumptions on $\gamma(\hf)$.

Note that $\gamma(\ha_k(n,k))=1$ and $\gamma(\ha_{k-1}(n,k))=n-k$. In a still unpublished paper Kupavskii \cite{Ku2} proves best possible bounds in the range $2\leq \gamma(\hf)<n-k$, thereby improving earlier results by Han and Kohayakawa \cite{HK}.

Another important tool to tackle intersecting families is shifting that can be tracked back to Erd\H{o}s-Ko-Rado \cite{ekr}. We are going to give the formal definition in Section 2 but let us define here the ``end product" of shifting. Let $(x_1,\ldots,x_k)$ denote the set $\{x_1,\ldots,x_k\}$ where we know or want to stress that $x_1<\ldots<x_k$.
Define the {\it shifting partial order} $\prec$ by setting $(a_1,\ldots,a_k)\prec (b_1,\ldots,b_k)$ iff $a_i\leq b_i$ for $1\leq i\leq k$. Then $\ha\subset \binom{[n]}{k}$ is called {\it initial} iff for all  $A,B\in \binom{[n]}{k}$, $A\prec B$ and $B\in \ha$ imply $A\in \ha$.

\begin{fact}\label{fact-1}
Let $n>3k-2$, $k\geq 2$. Suppose that $\hf\subset \binom{[n]}{k}$ is intersecting and initial. Then
\begin{align}\label{ineq-diversity}
\gamma(\hf)\leq \binom{n-3}{k-2}.
\end{align}
\end{fact}

For the proof let us recall that for $t\geq 1$, the families $\hf,\hg\subset 2^{[n]}$ are called cross $t$-intersecting if $|F\cap G|\geq t$ for all $F\in \hf$, $G\in \hg$.

\begin{fact}\label{fact-2}
Suppose that $\hf,\hg\subset 2^{[n]}$ are initial and cross-intersecting. Then $\hf(\bar{1})$ and $\hg(\bar{1})$ are cross 2-intersecting.
\end{fact}

\begin{proof}
Suppose for contradiction that for some $F\in \hf(\bar{1})$, $G\in \hg(\bar{1})$ and $2\leq j\leq n$, $F\cap G=\{j\}$ holds. Since $(F\setminus \{j\})\cup \{1\}=:F' \prec F$, $F'\in \hf$. However $F'\cap G=\emptyset$, a contradiction.
\end{proof}

\begin{proof}[Proof of Fact \ref{fact-1}]
In view of Fact \ref{fact-2},  $\hf(\bar{1})\subset \binom{[2,n]}{k}$ is 2-intersecting. For $n-1\geq 3(k-1)$ we may apply the Exact Erd\H{o}s-Ko-Rado Theorem (\cite{F78,Wi}) to infer $|\hf(\bar{1})|\leq \binom{(n-1)-2}{k-2}=\binom{n-3}{k-2}$.
\end{proof}

The above fact motivated the first author to conjecture in \cite{F17} that \eqref{ineq-diversity} holds for $n>3k$ for non-initial intersecting families as well. Counter-examples by Huang \cite{huang} and Kupavskii \cite{Ku1} showed that one needs to assume at least $n\geq (2+\sqrt{3})k$.

In the same paper Kupavskii showed that \eqref{ineq-diversity} holds for all intersecting families as long as $n>Ck$ where $C$ is a non-specified but very large constant. In \cite{F2020} the same was proved for $n> 72k$. The main result of the present paper is the following.

\begin{thm}\label{thm-main0}
Suppose that $n> 36k$, $\hf\subset \binom{[n]}{k}$ is intersecting. Then
\begin{align}\label{ineq-main1}
\gamma(\hf)\leq \binom{n-3}{k-2}.
\end{align}
\end{thm}

The main tool in the proof of \eqref{ineq-main1} is the following.

\begin{thm}\label{thm-main1}
Suppose that $\hf\subset \binom{[n]}{k}$ is an intersecting family with $n\geq 2k$ and $|\hf|\geq 36\binom{n-3}{k-3}$. Then
\begin{align}\label{ineq-main2}
\varrho(\hf)>\frac{1}{2}.
\end{align}
\end{thm}

This and the next result provide considerable improvements on the bounds in \cite{F2020}.

\begin{thm}\label{thm-main2}
Suppose that $0<\varepsilon\leq\frac{1}{24}$, $n\geq \frac{k}{\varepsilon}$, $\hf\subset \binom{[n]}{k}$ is an intersecting family with $|\hf|\geq 36\binom{n-3}{k-3}$. Then $\varrho(\hf)>\frac{2}{3}-\varepsilon$.
\end{thm}

We should mention that the family $\widetilde{\hht} =\{T\in \binom{[n]}{k}\colon |T\cap [3]|=2\}$ satisfies $|\widetilde{\hht}|=3\binom{n-3}{k-2}$ and $\varrho(\widetilde{\hht})=\frac{2}{3}$ showing that $\frac{2}{3}$ is best possible.

Let us present some  results and notations that are needed in our proofs. The first one is an inequality concerning cross-intersecting families.

\begin{prop}[\cite{FW2022}]\label{prop-3.4}
Suppose that $\ha,\hb\subset \binom{[m]}{\ell}$ are cross-intersecting and $m\geq 2\ell$,
$\min\{|\ha|,|\hb|\}\geq \binom{m-3}{\ell-3}+\binom{m-4}{\ell-3}$. Then
\begin{align}
|\ha|+|\hb|\leq 2\binom{m-1}{\ell-1}.
\end{align}
\end{prop}

Define the family of {\it transversals}, $\hht(\hf)$ by
\[
\hht(\hf) =\left\{T\subset [n]\colon T\cap F\neq \emptyset \mbox{ for all } F\in \hf\right\}.
\]
Set $\hht^{(t)}(\hf) = \{T\in \hht(\hf)\colon |T|=t\}$.

For $P\subset Q\subset [n]$, let
\[
\hf(P,Q) = \left\{F\setminus Q\colon F\cap Q=P, F\in\hf\right\}\subset 2^{[n]\setminus Q}.
\]
We also use $\hf(\bar{Q})$ to denote $\hf(\emptyset, Q)$. For $\hf(\{i\},Q)$ we simply write  $\hf(i,Q)$.


Let us recall the following inequalities concerning cross $t$-intersecting families.

\begin{thm}[\cite{F78}]\label{thm-tintersecting}
Suppose that $\ha,\hb\subset \binom{[n]}{k}$ are cross $t$-intersecting, $|\ha|\leq |\hb|$. Then either
\begin{align}
|\hb|\leq \binom{n}{k-t} \mbox{ \rm or }\label{ineq-tinterhb}\\[5pt]
|\ha| \leq \binom{n}{k-t-1}.\label{ineq-tinterha}
\end{align}
\end{thm}

We say that the cross-intersecting families $\hf,\hg$ are {\it saturated} or form {\it  a saturated pair} if adding an extra $k$-set to either of the families would destroy the cross-intersecting property.

For $\hf\subset \binom{[n]}{k}$ define the {\it shadow} of $\hf$,
\[
\partial\hf =\left\{G\colon |G|=k-1, \exists F\in \hf, G\subset F\right\}.
\]

\section{Shifting ad extremis}

In this section, we recall a modified shifting technique called shifting ad extremis, which was introduced in \cite{F2022}, \cite{FW2022-3}.

Recall the shifting operation as follows. Let $1\leq i< j\leq n$, $\hf\subset{[n]\choose k}$. Define
$$S_{ij}(\hf)=\left\{S_{ij}(F)\colon F\in\hf\right\},$$
where
$$S_{ij}(F)=\left\{
                \begin{array}{ll}
                  (F\setminus\{j\})\cup\{i\}, & j\in F, i\notin F \text{ and } (F\setminus\{j\})\cup\{i\}\notin \hf; \\[5pt]
                  F, & \hbox{otherwise.}
                \end{array}
              \right.
$$

 Let $\hf\subset\binom{[n]}{k}$, $\hg\subset\binom{[n]}{\ell}$ be families having certain properties (e.g., intersecting, cross $t$-intersecting) that are maintained by simultaneous shifting and certain properties (e.g., $\tau(\hf)\geq s$, $\varrho(\hg)\leq c$) that might be destroyed by shifting. Let $\hp$ be the collection of the latter properties that we want to maintain.

Define the quantity
\[
w(\hg) =\sum_{G\in \hg} \sum_{i\in G} i.
\]
Obviously $w(S_{ij}(\hg))\leq w(\hg)$ for $1\leq i<j\leq n$ with strict inequality unless $S_{ij}(\hg)=\hg$.

\begin{defn}\label{defn-2.1}
Suppose that $\hf\subset \binom{[n]}{k}$, $\hg\subset \binom{[n]}{\ell}$ are families having property $\hp$. We say that $\hf$ and $\hg$ have been {\it shifted ad extremis} with respect to $\hp$ if  $S_{ij}(\hf)=\hf$ and $S_{ij}(\hg)=\hg$ for every pair $1\leq i<j\leq n$ whenever $S_{ij}(\hf)$ and $S_{ij}(\hg)$ also have property $\hp$.
\end{defn}

Let us illustrate Definition \ref{defn-2.1} with the case $\hf=\hg\subset\binom{[n]}{k}$, intersecting and $\hp=\{\varrho(\hf)\leq \frac{1}{2}\}$. In this case $\hf$ is shifted ad extremis if $S_{ij}(\hf)=\hf$ for all $1\leq i<j\leq n$ unless $\varrho(S_{ij}(\hf))>\frac{1}{2}$. That is, we can define a graph $\mathds{H}=\mathds{H}_{\hp}(\hf)$ of shift-resistant pairs $(i,j)$ such that
\begin{align*}
&\varrho(S_{ij}(\hf))>\frac{1}{2} \mbox{ for } (i,j)\in  \mathds{H} \mbox{ and } S_{ij}(\hf)=\hf \mbox{ for }(i,j)\notin \mathds{H}.
\end{align*}
Note that $\hf$ is initial if and only if $\mathds{H}$ is empty.

We can obtain shifted ad extremis families by the following shifting ad extremis process. Let $\hf$,$\hg$ be cross-intersecting families with property $\hp$. Apply the shifting operation $S_{ij}$, $1\leq i<j\leq n$, to $\hf,\hg$ simultaneously and continue as long as the property $\hp$ is maintained. Recall that the shifting operation preserves the cross-intersecting property (cf. \cite{F87}). By abuse of notation, we keep denoting the current families by $\hf$ and $\hg$ during the shifting process. If $S_{ij}(\hf)$ or $S_{ij}(\hg)$ does not have property $\hp$, then we do not apply  $S_{ij}$ and choose a different pair $(i',j')$. However we keep returning to previously failed pairs $(i,j)$, because it might happen that at a later stage in the process $S_{ij}$ does not destroy property $\hp$ any longer. Note that the quantity $w(\hf)+w(\hg)$ is a positive integer and it decreases strictly in  each step. Eventually we shall arrive at families that are shifted ad extremis with respect to $\hp$.

Let $\hf\subset\binom{[n]}{k}$. If for $D\subset [n]$, $\hf(D)=\binom{[n]\setminus D}{k-|D|}$ then we say that $D$ is {\it full} in $\hf$ or $\hf(D)$ is {\it full}.

Let us prove a result concerning pairs of cross-intersecting families.

\begin{prop}\label{prop-2.10}
Let $n,k,\ell$ be positive integers, $n>k+\ell$. Suppose that $\hf\subset \binom{[n]}{k}$, $\hg \subset \binom{[n]}{\ell}$ are saturated non-trivial cross-intersecting. Suppose also that $\hf$ and $\hg$ are shifted ad extremis for $\hp=\{\tau(\hf)\geq 2,\tau(\hg)\geq 2\}$ but not (both) initial.
 Then one can find four distinct vertices $x_1,x_2,y_1,y_2$ such that either $\hf$ or $\hg$ (call it $\hj$) satisfies (i), (ii) and (iii).
 \begin{itemize}
   \item[(i)] $S_{x_iy_j}(\hj)$ is a star for $i=1,2;j=1,2$.
   \item[(ii)] $\hj(\overline{x_1},\overline{x_2})\neq\emptyset$, $\hj(\overline{y_1},\overline{y_2})\neq\emptyset$.
   \item[(iii)] At most one of $\{x_1,x_2\}$ and $\{y_1,y_2\}$ is full in $\hj$.
 \end{itemize}
\end{prop}

 Before the proof let us note that once the statement is proved we can assume without loss of generality $x_1<x_2$, $y_1<y_2$, $x_1<y_1$. Note also that by saturatedness, (ii) and (iii) imply that at most one of $\{x_1,x_2\}$ and $\{y_1,y_2\}$ is shift-resistant.

\begin{proof}
 Since $\hf$, $\hg$ are shifted ad extremis for $\hp=\{\tau(\hf)\geq 2,\tau(\hg)\geq 2\}$, we see that for every $1\leq i<j\leq n$ either (a) or (b) or (c) occurs.

\begin{itemize}
  \item[(a)] $S_{ij}(\hf)=\hf$ and $S_{ij}(\hg)=\hg$;
  \item[(b)] $S_{ij}(\hf)$ is a star;
  \item[(c)] $S_{ij}(\hg)$ is a star.
\end{itemize}

Recall that $\hf$, $\hg$ are not both initial. Since the statement is symmetric in $\hf$ and $\hg$ assume for the proof that there exist $1\leq x_1<y_1\leq n$ such that $(x_1,y_1)$ is of type (c). Then $\hg(\overline{x_1},\overline{y_1})=\emptyset$, $\hg(x_1)\cap \hg(y_1)=\emptyset$ and by saturatedness $\hf(x_1,y_1)$ is full. Choose $K,L\in \hg$ with $K\cap (x_1,y_1)=\{x_1\}$, $L\cap (x_1,y_1)=\{y_1\}$ and  $|K\cap L|$ maximal. By $\hg(x_1)\cap \hg(y_1)=\emptyset$ we infer $|K\cap L|\leq \ell-2$. Choose $x_2\in K\setminus L$, $x_2\neq x_1$; $y_2\in L\setminus K$, $y_2\neq y_1$. Then $(x_2,y_2)$ is not of type (a). Indeed, if $x_2<y_2$ and $(x_2,y_2)$ is of type (a), then $S_{x_2y_2}(\hg)=\hg$ implies $(L\setminus\{y_2\})\cup\{x_2\}=:L'\in \hg$. But $|K\cap L'|=|K\cap L|+1$, a contradiction. Similarly, $y_2<x_2$ and $S_{y_2x_2}(\hg)=\hg$ would imply $(K\setminus\{x_2\})\cup \{y_2\}=:K'\in \hg$. Again, $|K'\cap L|=|K\cap L|+1$ provides the contradiction.

If $(x_2,y_2)$ is of type (b), i.e., $\hf(\overline{x_2},\overline{y_2})=\emptyset$, then $\hg(x_2,y_2)$ would be full. But then $\hf(x_1,y_1)$, $\hg(x_2,y_2)$ are not cross-intersecting, a contradiction. Consequently, $(x_2,y_2)$ is of type (c), that is, $\hf(x_2,y_2)$ is full, $\hg(\overline{x_2},\overline{y_2})=\emptyset$.

Let us prove that $\{x_1,y_2\}$ and $\{x_2,y_1\}$ are shift-resistant. By symmetry we consider the case $x_1<y_2$. (Recall $\{x_1,x_2\}\subset K$, $\{y_1,y_2\}\subset L$, $K,L\in \hg$). If $S_{x_1y_2}(\hg)=\hg$ then $(L\setminus \{y_2\})\cup \{x_1\}=:L''\in \hg$ but $L''\cap\{x_2,y_2\}=\emptyset$ in contradiction with $\hg(\overline{x_2},\overline{y_2})=\emptyset$.
Thus $(x_1,y_2)$ is not of type (a). Similarly $\{x_2,y_1\}$ is not of type (a).

We showed that both $\{x_1,y_2\}$ and $\{x_2,y_1\}$ are of type (b) or (c). By cross-intersection their types could not be different. Set $Z=\{x_1,x_2,y_1,y_2\}$. If both of them are of type (b) then $\hg(x_1,y_2)$ and $\hg(x_2,y_1)$ are full. Consequently for all $T\in \binom{Z}{3}$, $T$ is full in $\hg$ too. Let $R\in \binom{[n]\setminus Z}{k-3}$ be arbitrary. Then $R\cup Z\setminus \{x_1\}$ and $R\cup Z\setminus \{y_1\}$ are in $\hg$ showing $R\cup\{x_2,y_2\}\in \hg(x_1)\cap \hg(y_1)$, a contradiction. Thus $\{x_1,y_1\}$, $\{x_1,y_2\}$, $\{x_2,y_1\}$ and $\{x_2,y_2\}$ are all of type (c) and (i) holds.

Let us show that $\hg(\overline{x_1},\overline{x_2})\neq\emptyset$. Indeed, if $\hg(\overline{x_1},\overline{x_2})=\emptyset$ then by saturatedness $\hf(x_1,x_2)$ is full. However that would contradict $F\cap L\neq \emptyset$ for all $F\in \hf$. Similarly $\hg(\overline{y_1},\overline{y_2})\neq\emptyset$ and (ii) holds.


If  both $\{x_1,x_2\}$ and $\{y_1,y_2\}$  are full in $\hg$, then all 3-subsets of $Z$ are full. Let $R\in \binom{[n]\setminus Z}{k-3}$. Then $R\cup Z\setminus \{x_1\}$ and $R\cup Z\setminus \{y_1\}$ are in $\hg$ showing $R\cup\{x_2,y_2\}\in \hg(x_1)\cap \hg(y_1)$, a contradiction. Thus (iii) holds and the proposition is proven.
\end{proof}

\begin{cor}
Suppose that $\hf\subset\binom{[n]}{k}$, $\hg\subset\binom{[n]}{\ell}$ and $\hf,\hg$ are  non-trivial cross-intersecting with
\begin{align}\label{ineq-2.5}
|\hf|>\binom{n-2}{k-2}+\binom{n-4}{k-2}-\binom{n-\ell-4}{k-2},\ |\hg|>\binom{n-2}{\ell-2}+\binom{n-4}{\ell-2}-\binom{n-k-4}{\ell-2}.
\end{align}
Then there are non-trivial, initial cross-intersecting families with the same size.
\end{cor}

\begin{proof}
Shift $\hf,\hg$ ad extremis with respect to $\{\tau(\hf)\geq 2, \tau(\hg)\geq 2\}$.  Assume indirectly that $\hf,\hg$ are not both initial. We may further assume that $\hf,\hg$ form a saturated pair. Then by Proposition \ref{prop-2.10} there exist $x_1,x_2,y_1,y_2$ such that either $\hf$ or $\hg$ satisfies  (i), (ii) and (iii). Since the statement is symmetric in $\hf$ and $\hg$, assume that $\hg$ satisfies  (i), (ii) and (iii) of Proposition \ref{prop-2.10}.

Let $Z=\{x_1,x_2,y_1,y_2\}$. By (i) $(x_1,y_1),(x_2,y_1)$, $(x_1,y_2),(x_2,y_2)$ are all full in $\hf$. By cross-intersection, we see that either $\{x_1,x_2\}\subset G$ or $\{y_1,y_2\}\subset G$ for every $G\in \hg$. Let
 \[
 \hg_r =\{G\in \hg\colon |G\cap Z|=r\},\ r=2,3,4.
 \]
 For any $Q\in \binom{[n]\setminus Z}{\ell-2}$ and $P\in \binom{Z}{2}$, if $P\cup Q\in \hg$ then $P$ has to be $(x_1,x_2)$ or $(y_1,y_2)$. Define
\begin{align*}
&\hq_x=\left\{Q\in \binom{[n]\setminus Z}{\ell-2}\colon Q\cup \{x_1,x_2\}\in \hg,\ Q\cup \{y_1,y_2\}\notin \hg\right\},\\[5pt]
&\hq_y=\left\{Q\in \binom{[n]\setminus Z}{\ell-2}\colon Q\cup \{x_1,x_2\}\notin \hg,\ Q\cup \{y_1,y_2\}\in \hg\right\},\\[5pt]
&\hq_{xy}=\left\{Q\in \binom{[n]\setminus Z}{\ell-2}\colon Q\cup \{x_1,x_2\}\in \hg,\ Q\cup \{y_1,y_2\}\in \hg\right\}.
\end{align*}
By (iii) we infer that at least one of $\{x_1,x_2\}$ and $\{y_1,y_2\}$ is not full in $\hg$. It follows that $\hf(\overline{x_1},\overline{x_2})\cup \hf(\overline{y_1},\overline{y_2})$ is non-empty. Let $F_0\in \hf(\overline{x_1},\overline{x_2})\cup \hf(\overline{y_1},\overline{y_2})$. Since $Q\cap F_0\neq \emptyset$ for every $Q\in \hq_{xy}$, we infer that
\[
|\hq_{xy}|\leq \binom{n-4}{\ell-2} - \binom{n-4-|F_0\setminus Z|}{\ell-2}\leq \binom{n-4}{\ell-2} - \binom{n-4-k}{\ell-2}.
\]
Note that $|\hq_x|+|\hq_y|+|\hq_{xy}| \leq \binom{n-4}{\ell-2}$. It follows that
\begin{align}\label{ineq-2.3}
|\hg_2|=2|\hq_{xy}|+|\hq_x|+|\hq_y| \leq 2\binom{n-4}{\ell-2} - \binom{n-4-k}{\ell-2}.
\end{align}

\begin{claim}\label{claim-2.14}
For $R\in \binom{[n]\setminus Z}{\ell-3}$ there are  at most two choices for $S\in \binom{Z}{3}$ with $R\cup S\in \hg$.
\end{claim}

\begin{proof}
If there are three choices for $S\in \binom{Z}{3}$ with $R\cup S\in \hg$, then  we may choose $S_1,S_2$ such that $S_1\cap S_2= \{x_1,y_1\}$ or $\{x_2,y_2\}$. Without loss of generality assume that $S_1=\{x_1,y_1,x_2\}$, $S_2=\{x_1,y_1,y_2\}$. Then $R\cup\{x_1,y_1\}\in \hg(x_2)\cap \hg(y_2)$, contradiction.
\end{proof}

By Claim \ref{claim-2.14} we see that
\begin{align}\label{ineq-2.4}
|\hg_3|\leq 2\binom{n-4}{\ell-3}.
\end{align}
Combining \eqref{ineq-2.3} and \eqref{ineq-2.4}, we obtain that
\begin{align*}
|\hg|=|\hg_2|+|\hg_3|+|\hg_4|&\leq 2\binom{n-4}{\ell-2} - \binom{n-4-k}{\ell-2}+2\binom{n-4}{\ell-3}+\binom{n-4}{\ell-4}\\[5pt]
&= \binom{n-2}{\ell-2}+\binom{n-4}{\ell-2}-\binom{n-k-4}{\ell-2},
\end{align*}
contradicting \eqref{ineq-2.5}.
\end{proof}

\section{Conditions guaranteeing $\varrho(\hf)> \frac{1}{2}$}

In this section, we prove Theorem \ref{thm-main1}, which leads to an improved bound for the maximum diversity result in \cite{F2020}.

Let us prove the following lemma.

\begin{lem}
Let $\hf,\hg\subset \binom{[n]}{k}$ be cross-intersecting. Suppose that for some $\{x,y\}\in \binom{[n]}{2}$,
\begin{align}
&|\hf(x,y)|\geq \binom{n-3}{k-3}+\binom{n-4}{k-3}+\binom{n-6}{k-4}. \mbox{ Then} \label{ineq-hfxy}\\[5pt]
&|\hg(\bar{x},\bar{y})|\leq \binom{n-5}{k-3}+\binom{n-6}{k-3}.\label{ineq-hgxy}
\end{align}
\end{lem}

\begin{proof}
Without loss of generality let $(x,y)=(n-1,n)$. Then $\hf(x,y)\subset\binom{[n-2]}{k-2}$, $\hg(\bar{x},\bar{y})\subset\binom{[n-2]}{k}$ and $\hf(x,y),\hg(\bar{x},\bar{y})$ are cross-intersecting. By Hilton's lemma  the same holds for the lexicographic initial families. By \eqref{ineq-hfxy} $\hl(n-2,k-2,|\hf(x,y)|)$ contains
\[
\left\{\bar{F}\in \binom{[n-2]}{k-2}\colon \bar{F}\cap\{1,2\}\neq \emptyset\right\}\bigcup \left\{\bar{F}\in \binom{[n-2]}{k-2}\colon \{3,4\}\subset \bar{F}\right\}.
\]
By cross-intersection,
\[
\hl(n-2,k,|\hg(\bar{x},\bar{y})|)\subset \left\{G\in \binom{[n-2]}{k}\colon \{1,2\}\subset G \mbox{ and } G\cap \{3,4\}\neq \emptyset\right\}.
\]
These are $\binom{n-5}{k-3}+\binom{n-6}{k-3}$ sets. Hence $|\hg(\bar{x},\bar{y})|\leq \binom{n-5}{k-3}+\binom{n-6}{k-3}$.
\end{proof}

\begin{cor}\label{cor-1}
If $\hf=\hg$ then \eqref{ineq-hfxy} implies $\varrho(\hf)>\frac{1}{2}$.
\end{cor}
\begin{proof}
Note that by \eqref{ineq-hfxy} and \eqref{ineq-hgxy} we have $|\hf(\bar{x},\bar{y})|<|\hf(x,y)|$. By symmetry assume $|\hf(x,\bar{y})|\leq |\hf(\bar{x},y)|$. Then
\[
|\hf(\bar{x},\bar{y})|+|\hf(x,\bar{y})|<|\hf(x,y)|+|\hf(\bar{x},y)|.
\]
That is, $|\hf(\bar{y})|<|\hf(y)|$ and  $\varrho(\hf)>\frac{1}{2}$ follows.
\end{proof}

The next statement is well-known. Let us include the simple proof.

\begin{prop}
Let $\hf\subset \binom{[n]}{k}$ be an initial family. Then
\begin{align}\label{ineq-shifted}
\partial \hf(\bar{1})\subset \hf(1).
\end{align}
\end{prop}

\begin{proof}
Suppose that $E\subset F\in \hf(\bar{1})$ and $E=F\setminus \{j\}$. Then by initiality $E\cup\{1\}\in \hf$, i.e., $E\in \hf(1)$.
\end{proof}

 For $\hf\subset \binom{[n]}{k}$ and $P\in \binom{[n]}{2}$, define
\[
\hf_P=\{F\in \hf\colon F\cap P\neq\emptyset\}.
\]
For the proofs  of Theorems \ref{thm-main1} and \ref{thm-main2}, we need the following lemma.

\begin{lem}\label{lem-3.5}
Let $\hf\subset \binom{[n]}{k}$ be an intersecting family with $|\hf(P)|\leq M$ for every $P\in \binom{[n]}{2}$. Suppose that $R,Q\in \binom{[n]}{2}$ and $R\cap Q=\emptyset$. Then
\begin{align}\label{ineq-FQR-general}
|\hf_R\cap \hf_Q|\leq \max\left\{3M+\binom{n-7}{k-5}+\binom{n-8}{k-5}, 2M+2\binom{n-5}{k-3}\right\}.
\end{align}
Moreover, if $M\geq 2\binom{n-5}{k-3}$ then
\begin{align}\label{ineq-FQR-general2}
|\hf_R\cap \hf_Q|\leq 3M+\binom{n-7}{k-5}+\binom{n-8}{k-5}.
\end{align}
\end{lem}

\begin{proof}
Set $\hb=\hf_R\cap \hf_Q$ and define the partition $\hb=\hb_2\cup \hb_3\cup \hb_4$ via
\[
\hb_i =\left\{B\in \hb\colon |B\cap (R\cup Q)|=i\right\}.
\]
For $x\in R$, $y\in Q$ define
\[
\tilde{\hf}(x,y)=\left\{F\in \hf\colon x\in F \mbox{ and }y\in F\right\}.
\]
Note that if $B\in \hb_4$ then $B\in \tilde{\hf}(x,y)$ for all four choices $x\in R$, $y\in Q$ and if $B\in \hb_3$ then $B\in \tilde{\hf}(x,y)$ for two choices $x\in R$, $y\in Q$. Thus,
\[
|\hb_2|+2|\hb_3|+4|\hb_4|=\sum_{x\in R}\sum_{y\in Q} |\tilde{\hf}(x,y)|.
\]
Therefore,
\begin{align}\label{ineq-hb}
2|\hb|=2(|\hb_2|+|\hb_3|+|\hb_4|)\leq \sum_{x\in R}\sum_{y\in Q}|\tilde{\hf}(x,y)|+|\hb_2|\leq 4M+|\hb_2|.
\end{align}

Assume $R=(x_1,x_2)$, $Q=(y_1,y_2)$. Note that $\hb_2(x_i,y_j)\subset\binom{[n]\setminus (R\cup Q)}{k-2}$ for $1\leq i,j\leq 2$. Since $\hf$ is intersecting, $(\hb_2(x_1,y_1),\hb_2(x_2,y_2))$ and $(\hb_2(x_1,y_2),\hb_2(x_2,y_1))$ are cross-intersecting pairs. If either of the two families consists of less than $\binom{n-7}{k-5}+\binom{n-8}{k-5}$ sets, then
\[
|\hb_2(x_1,y_1)|+|\hb_2(x_2,y_2)|\leq M+\binom{n-7}{k-5}+\binom{n-8}{k-5}.
\]
If $|\hb_2(x_i,y_i)|\geq \binom{n-7}{k-5}+\binom{n-8}{k-5}$ for both $i=1,2$, then by applying Proposition \ref{prop-3.4} with $m=n-4$ and $\ell=k-2$ we infer
\[
|\hb_2(x_1,y_1)|+|\hb_2(x_2,y_2)|\leq 2\binom{n-5}{k-3}.
\]
Thus
\begin{align}\label{hb2-1}
|\hb_2(x_1,y_1)|+|\hb_2(x_2,y_2)|\leq \max\left\{M+\binom{n-7}{k-5}+\binom{n-8}{k-5}, 2\binom{n-5}{k-3}\right\}.
\end{align}
Similarly,
\begin{align}\label{hb2-2}
|\hb_2(x_1,y_2)|+|\hb_2(x_2,y_1)|\leq \max\left\{M+\binom{n-7}{k-5}+\binom{n-8}{k-5}, 2\binom{n-5}{k-3}\right\}.
\end{align}
Note that $|\hb_2|=|\hb_2(x_1,y_1)|+|\hb_2(x_2,y_2)|+|\hb_2(x_1,y_2)|+|\hb_2(x_2,y_1)|$.
Substituting \eqref{hb2-1} and \eqref{hb2-2} into \eqref{ineq-hb}, we conclude that \eqref{ineq-FQR-general} holds. If $M\geq 2\binom{n-5}{k-3}$ then \eqref{ineq-FQR-general2} follows from \eqref{ineq-FQR-general} directly.
\end{proof}

\begin{proof}[Proof of Theorem \ref{thm-main1}]
Arguing indirectly assume that $\hf\subset \binom{[n]}{k}$ is  intersecting, $|\hf|\geq 36\binom{n-3}{k-3}$ and $\varrho(\hf)\leq \frac{1}{2}$. Without loss of generality suppose that
$\hf$ is shifted ad extremis with respect to $\{\varrho(\hf)\leq \frac{1}{2}\}$ and let $\mathds{H}$ be the graph formed by the shift-resistant pairs.
\begin{claim}
$\mathds{H}\neq \emptyset$.
\end{claim}
\begin{proof}
If $\mathds{H}=\emptyset$ then $\hf$ is initial. By Fact \ref{fact-2}, $\hf(\bar{1})$ is 2-intersecting. Then  by the Katona Intersecting Shadow Theorem \cite{Katona}, we see $|\partial\hf(\bar{1})|>|\hf(\bar{1})|$. By \eqref{ineq-shifted} $\partial \hf(\bar{1})\subset \hf(1)$. Thus
\[
|\hf(1)|\geq |\partial\hf(\bar{1})|>|\hf(\bar{1})|,
\]
it follows that $\varrho(\hf)>\frac{1}{2}$, contradicting our assumption.
\end{proof}

By Corollary \ref{cor-1} and \eqref{ineq-hfxy}, we may assume that for every $P\in \binom{[n]}{2}$,
\begin{align}\label{ineq-hfpupbd}
|\hf(P)|< \binom{n-3}{k-3}+\binom{n-4}{k-3}+\binom{n-6}{k-4}< 2\binom{n-3}{k-3}.
\end{align}
Note that for any $(i,j)\in \mathds{H}$ we have
\begin{align}\label{ineq-sijhf}
|S_{ij}(\hf)(i)|>\frac{1}{2}|\hf|.
\end{align}
It is equivalent to
\begin{align}\label{ineq-sijshift}
|\hf(i,j)|+|\hf(i,\bar{j})\cup \hf(\bar{i},j)|> \frac{1}{2}|\hf|.
\end{align}

\begin{claim}
$\mathds{H}$ does not contain three pairwise disjoint edges.
\end{claim}
\begin{proof}
Suppose for contradiction that $(a_1,b_1), (a_2,b_2),(a_3,b_3)$ are pairwise disjoint and
\[
\varrho(S_{a_rb_r}(\hf))>\frac{1}{2} \mbox{ for } r=1,2,3.
\]
Let $\hg_r= \hf_{\{a_r,b_r\}}$, $r=1,2,3$. Then by \eqref{ineq-sijshift} we have for $r=1,2,3$
\begin{align*}
|\hg_r| = & |\hf(a_r,b_r)|+|\hf(a_r,\overline{b_r})\cup \hf(\overline{a_r},b_r)|>  \frac{1}{2}|\hf|.
\end{align*}
For $1\leq r<r'\leq 3$, applying Lemma \ref{lem-3.5} with $M=\binom{n-3}{k-3}+\binom{n-4}{k-3}+\binom{n-6}{k-4}>2\binom{n-5}{k-3}$, from \eqref{ineq-FQR-general2} we infer
\begin{align*}
|\hg_r\cap \hg_{r'}|&\leq 3M+\binom{n-7}{k-5}+\binom{n-8}{k-5}< 6\binom{n-3}{k-3} \leq \frac{1}{6}|\hf|.
\end{align*}
Then
\[
|\hg_1\cup\hg_2\cup \hg_3| > \frac{3}{2}|\hf|-|\hg_1\cap \hg_2|-|\hg_1\cap \hg_3|-|\hg_2\cap \hg_3|> |\hf|,
\]
a contradiction.
\end{proof}

Now we distinguish two cases.

\vspace{5pt}
{\noindent\bf Case 1.} $\mathds{H}$ has matching number one.
\vspace{5pt}

For notational convenience assume that $(n-1,n)\in \mathds{H}$. Since $\mathds{H}$ has matching number one, we infer that $S_{ij}(\hf)=\hf$ for any $(i,j)$ with $1\leq i<j\leq n-2$.
Define
\begin{align*}
\hg&=\{F\setminus\{n-1,n\}\colon F\in \hf,|F\cap\{n-1,n\}|=1\}\subset\binom{[n-2]}{k-1},\\[5pt]
\hh&=\{F\in \hf\colon F\subset [n-2]\}\subset \binom{[n-2]}{k}.
\end{align*}
The key to the proof is that $\hg,\hh$ are initial. Clearly, by \eqref{ineq-sijshift} and \eqref{ineq-hfpupbd} we have
\begin{align}\label{ineq-2.1}
|\hg|>\frac{1}{2}|\hf|-|\hf(n-1,n)|>\frac{1}{2}|\hf|-2\binom{n-3}{k-3}.
\end{align}
Now consider $\hg(\bar{1}),\hh(1)\subset \binom{[2,n-2]}{k-1}$. Since $\hf$ is intersecting, $\hg(\bar{1}),\hh(1)$ are cross-intersecting. Since \eqref{ineq-hfpupbd} implies
\begin{align}\label{ineq-2.2}
|\hg(1)|\leq |\hf(1,n-1)|+|\hf(1,n)|\leq 4\binom{n-3}{k-3},
\end{align}
by \eqref{ineq-2.1} and \eqref{ineq-2.2} it follows that
\[
|\hg(\bar{1})|=|\hg|-|\hg(1)|> \frac{1}{2}|\hf|-6\binom{n-3}{k-3} \geq 12\binom{n-3}{k-3}.
\]
Moreover, by $\hg(\bar{1},2)\subset \hf(2,n-1)\cup \hf(2,n)$ we infer
\[
|\hg(\bar{1},\bar{2})|=|\hg(\bar{1})|-|\hg(\bar{1},2)|>|\hg(\bar{1})|-|\hf(2,n-1)|- |\hf(2,n)| >8\binom{n-3}{k-3}.
\]
Since $\hg(\bar{1})$ and $\hh(1)$ are cross-intersecting and both initial, by Fact \ref{fact-2} we infer $\hg(\bar{1},\bar{2})$ and $\hh(1,\bar{2})$ are cross 2-intersecting.
Since $|\hg(\bar{1},\bar{2})|>\binom{n-4}{k-3}$,  by \eqref{ineq-tinterha} we see
\[
|\hh(1,\bar{2})|\leq \binom{n-4}{k-4}<\binom{n-3}{k-3}.
\]
Since $\hh$ is shifted and intersecting, we infer that $\hh(\bar{1},\bar{2})$ is 3-intersecting. Then by \eqref{ineq-tinterhb} we have
\[
|\hh(\bar{1},\bar{2})|\leq \binom{n-4}{k-3}<\binom{n-3}{k-3}.
\]
Note that by initiality $\hh(\bar{1},2)\subset \hh(1,\bar{2})$. Since \eqref{ineq-hfpupbd} implies $|\hh(1,2)|\leq |\hf(1,2)|< 2\binom{n-3}{k-3}$, we infer
\begin{align*}
|\hh|&= |\hh(1,2)|+|\hh(\bar{1},2)|+ |\hh(1,\bar{2})|+|\hh(\bar{1},\bar{2})|<5 \binom{n-3}{k-3}.
\end{align*}
Then $\varrho(\hf)\leq \frac{1}{2}$ implies
\[
|\hf(\overline{n-1},n)|\geq |\hf(\overline{n-1})|-|\hh|\geq \frac{1}{2}|\hf|-|\hh|>13 \binom{n-3}{k-3}
\]
and
\[
|\hf(n-1,\bar{n})|\geq |\hf(\bar{n})|-|\hh|\geq \frac{1}{2}|\hf|-|\hh|>13 \binom{n-3}{k-3}.
\]
Now
\[
|\hf(\overline{\{1,n-1\}},n)|\geq |\hf(\overline{n-1},n)|-|\hf(1,n)|> 11\binom{n-3}{k-3}
\]
and
\[
|\hf(\overline{\{1,n\}},n-1)|\geq |\hf(n-1,\overline{n})|-|\hf(1,n-1)|> 11\binom{n-3}{k-3}.
\]
But $\hf(\overline{\{1,n-1\}},n), \hf(\overline{\{1,n\}},n-1)\subset \binom{[2,n-2]}{k-1}$ are cross 2-intersecting, contradicting \eqref{ineq-tinterhb}.

\vspace{5pt}
{\noindent\bf Case 2.} $\mathds{H}$ has matching number two.
\vspace{5pt}

For notational convenience  assume that $(n-1,n),(n-3,n-2)\in \mathds{H}$. Since $\mathds{H}$ has matching number two, we infer that $S_{ij}(\hf)=\hf$ for any $(i,j)$ with $1\leq i<j\leq n-4$. Define $\ha,\hb \subset \binom{[n-4]}{k-1}$ as
\begin{align*}
\ha=& \{F\setminus \{n-1,n\}\colon F\in \hf, |F\cap \{n-1,n\}|=1, F\cap \{n-3,n-2\}=\emptyset\},\\[5pt]
\hb=&\{F\setminus\{n-3,n-2\}\colon F\in \hf, F\cap \{n-1,n\}=\emptyset, |F\cap \{n-3,n-2\}|=1\}.
\end{align*}
Then by \eqref{ineq-sijshift} and \eqref{ineq-hfpupbd} we have
\begin{align*}
|\ha|&> \frac{1}{2}|\hf|-|\hf(n-1,n)|-\sum_{i\in \{n-1,n\},j\in \{n-3,n-2\}}|\hf(i,j)|> 8\binom{n-3}{k-3},\\[5pt]
|\hb|&> \frac{1}{2}|\hf|-|\hf(n-3,n-2)|-\sum_{i\in \{n-1,n\},j\in \{n-3,n-2\}}|\hf(i,j)|> 8\binom{n-3}{k-3}.
\end{align*}
Since $\ha,\hb$ are initial, it follows that $\ha(\bar{1})$ and $\hb(\bar{1})$ are cross 2-intersecting. By symmetry assume $|\ha(\bar{1})|\leq |\hb(\bar{1})|$. By \eqref{ineq-tinterhb} we have
\[
|\ha(\bar{1})|\leq \binom{n-5}{k-3}<\binom{n-3}{k-3}
\]
and thereby
\[
|\ha(1)|=|\ha|-|\ha(\bar{1})|> 7\binom{n-3}{k-3}.
\]
It follows that
\[
\mbox{either } |\ha(1,n-1)|\geq 3.5\binom{n-3}{k-3} \mbox{ or }|\ha(1,n)|\geq 3.5\binom{n-3}{k-3},
\]
contradicting \eqref{ineq-hfpupbd}.
\end{proof}

Combining Theorem \ref{thm-main1} with $n\geq 36k$, $|\hf|\geq \binom{n-2}{k-2}>36\binom{n-3}{k-3}$ and repeating the proof of Theorem 5 in \cite{F2020}, one can check that Theorem \ref{thm-main0} follows. As a matter of fact, it is explicitly stated in \cite{F2020} that knowing the corresponding $\varrho(\hf)>\frac{1}{2}$ result with $n\geq ck$, $c\geq 10$ would imply the diversity bound for $n>ck$. That is, improving 36 in Theorem \ref{thm-main1} would automatically improve 36 in Theorem \ref{thm-main0} as well.

\section{Conditions guaranteeing $\varrho(\hf)>\frac{2}{3}-o(1)$}

In this section, by using a diversity result for cross-intersecting families in \cite{FKu2021} and following  a similar but simpler approach as  in Section 3, we prove Theorem \ref{thm-main2}.

Let us state the diversity result for cross-intersecting families as follows.

\begin{thm}[\cite{FKu2021}]\label{thm-3.1}
Let $n\geq 2k$. Suppose that $\ha,\hb\subset \binom{n}{k}$ are cross-intersecting. If
\[
|\ha|,|\hb|> \binom{n-1}{k-1}-\binom{n-u-1}{k-1}+\binom{n-u-1}{k-u} \mbox{ with }3\leq u\leq k,
\]
then
\begin{align}\label{ineq-diversity2}
\gamma(\ha),\gamma(\hb) <\binom{n-u-1}{k-u},
\end{align}
moreover, both families have the same (unique) element of the largest degree.
\end{thm}

\begin{lem}
Suppose that $\hf\subset \binom{[n]}{k}$ is an intersecting family and for some $\{x,y\}\in \binom{[n]}{2}$
\begin{align}
&|\hf(x,y)|\geq \binom{n-3}{k-3}+\binom{n-4}{k-3}+\binom{n-5}{k-3}+\binom{n-7}{k-4}. \mbox{ Then} \label{ineq-hfxy2}\\[5pt]
&|\hf(\bar{x},\bar{y})|\leq \binom{n-6}{k-4}+\binom{n-7}{k-4}.\label{ineq-hgxy2}
\end{align}
\end{lem}

\begin{proof}
For notational convenience  let $(x,y)=(n-1,n)$. Then $\hf(x,y)\subset\binom{[n-2]}{k-2}$, $\hf(\bar{x},\bar{y})\subset\binom{[n-2]}{k}$ and $\hf(x,y),\hf(\bar{x},\bar{y})$ are cross-intersecting. By Hilton's lemma the same holds for the lexicographic initial families. By \eqref{ineq-hfxy2} $\hl(n-2,k-2,|\hf(x,y)|)$ contains
\[
\left\{\bar{F}\in \binom{[n-2]}{k-2}\colon \bar{F}\cap\{1,2,3\}\neq \emptyset\right\}\bigcup \left\{\bar{F}\in \binom{[n-2]}{k-2}\colon \bar{F}\cap \{1,2,3\}=\emptyset, \{4,5\}\subset \bar{F}\right\}.
\]
By cross-intersection,
\[
\hl(n-2,k,|\hf(\bar{x},\bar{y})|)\subset \left\{F\in \binom{[n-2]}{k}\colon \{1,2,3\}\subset F\mbox{ and } F\cap \{4,5\}\neq \emptyset\right\}.
\]
These are $\binom{n-6}{k-4}+\binom{n-7}{k-4}$ sets. Hence $|\hf(\bar{x},\bar{y})|\leq \binom{n-6}{k-4}+\binom{n-7}{k-4}$.
\end{proof}

\begin{lem}\label{cor-2.5}
Let $0<\varepsilon\leq \frac{1}{6}$, $|\hf|\geq 12\binom{n-3}{k-3}$ and $n\geq \frac{k}{\varepsilon}$. Suppose that \eqref{ineq-hfxy2} holds for $\hf$ and some $\{x,y\}\in \binom{[n]}{2}$. Then $\varrho(\hf)>\frac{2}{3}-\varepsilon$.
\end{lem}

\begin{proof}
By symmetry  assume that $|\hf(x,\bar{y})|\geq |\hf(\bar{x},y)|$. By \eqref{ineq-hgxy2} we infer
\[
|\hf(\bar{x},\bar{y})|\leq \binom{n-6}{k-4}+\binom{n-7}{k-4}< 2\binom{n-4}{k-4}\leq \frac{2k}{ n}\binom{n-3}{k-3} \leq \frac{\varepsilon}{6}|\hf|.
\]
It follows that
\begin{align}\label{ineq-3.1}
|\hf(x,y)|+|\hf(x,\bar{y})|+|\hf(\bar{x},y)|&= |\hf|-|\hf(\bar{x},\bar{y})|> \left(1-\frac{\varepsilon}{6}\right)|\hf|.
\end{align}

If $|\hf(x,y)|\geq |\hf(x,\bar{y})|\geq |\hf(\bar{x},y)|$ or $|\hf(x,\bar{y})|\geq|\hf(x,y)| \geq |\hf(\bar{x},y)|$, then by \eqref{ineq-3.1}
\begin{align}\label{ineq-3.2}
|\hf(x)|=|\hf(x,y)|+|\hf(x,\bar{y})| &\geq \frac{2}{3}\left(|\hf(x,y)|+|\hf(x,\bar{y})|+|\hf(\bar{x},y)|\right)\nonumber\\[5pt]
&> \frac{2}{3}\left(1-\frac{\varepsilon}{6}\right)|\hf|\nonumber\\[5pt]
&> \left(\frac{2}{3}-\varepsilon\right)|\hf|.
\end{align}
Now we assume that $|\hf(x,\bar{y})|\geq |\hf(\bar{x},y)|\geq|\hf(x,y)|$.

If $|\hf(x,y)|+|\hf(x,\bar{y})| \geq \frac{2}{3}\left(1-\frac{\varepsilon}{6}\right)|\hf|$ then by \eqref{ineq-3.2} we are done as well. Thus we may also assume that
\begin{align}\label{ineq-3.3}
|\hf(x,y)|+|\hf(x,\bar{y})| < \frac{2}{3}\left(1-\frac{\varepsilon}{6}\right)|\hf|.
\end{align}
By \eqref{ineq-3.1} and \eqref{ineq-3.3} we infer
\[
|\hf(x,\bar{y})|\geq |\hf(\bar{x},y)| >  \frac{1}{3}\left(1-\frac{\varepsilon}{6}\right)|\hf| > \frac{1}{4}|\hf|\geq 3 \binom{n-3}{k-3}.
\]
It follows that
\begin{align}\label{ineq-3.4}
|\hf(x,\bar{y})|,|\hf(\bar{x},y)|> \binom{n-3}{k-2}-\binom{n-6}{k-2}+\binom{n-6}{k-4},
\end{align}
Since $\hf(x,\bar{y}),\hf(\bar{x},y)\subset \binom{[n]\setminus \{x,y\}}{k-1}$ are cross-intersecting,  by applying Theorem  \ref{thm-3.1} with $u=3$ we have
\[
\gamma(\hf(x,\bar{y})), \gamma(\hf(\bar{x},y)) <\binom{n-6}{k-4}<\binom{n-4}{k-4}.
\]
Let $i\in [n-2]$ be the element of the largest degree in both $\hf(x,\bar{y})$ and $\hf(\bar{x},y)$. Then
\begin{align*}
|\hf(i)|\geq|\hf(x,\bar{y},i)|+|\hf(\bar{x},y,i)|&\geq |\hf(x,\bar{y})|+|\hf(\bar{x},y)|-\gamma(\hf(x,\bar{y}))-\gamma(\hf(\bar{x},y))\\[5pt]
&> |\hf(x,\bar{y})|+|\hf(\bar{x},y)|-2\binom{n-4}{k-4}.
\end{align*}
Consequently,
\begin{align*}
|\hf(i)|&\geq \frac{2}{3}\left(|\hf(x,y)|+|\hf(x,\bar{y})|+|\hf(\bar{x},y)|\right)-2\binom{n-4}{k-4}\\[5pt]
 &\overset{\eqref{ineq-3.1}}> \frac{2}{3}\left(1-\frac{\varepsilon}{6}\right)|\hf|-\frac{\varepsilon}{ 6}|\hf|\\[5pt]
 &\geq \frac{2}{3}\left(1-\frac{5\varepsilon}{12}\right)|\hf|\\[5pt]
 &> \left(\frac{2}{3}-\varepsilon\right)|\hf|.
\end{align*}
\end{proof}

\begin{proof}[Proof of Theorem \ref{thm-main2}]
Arguing indirectly assume $\hf\subset \binom{[n]}{k}$ is  intersecting, $|\hf|\geq 36\binom{n-3}{k-3}$ and $\varrho(\hf)\leq \frac{2}{3}-\varepsilon$. Without loss of generality suppose that $\hf$ is shifted ad extremis with respect to $\{\varrho(\hf)\leq \frac{2}{3}-\varepsilon\}$ and let $\mathds{H}$ be the graph formed by the shift-resistant pairs.

By Lemma \ref{cor-2.5}, we may assume that for any $P\in \binom{[n]}{2}$,
\begin{align}\label{ineq-hfpupbd2}
|\hf(P)|&\leq \binom{n-3}{k-3}+\binom{n-4}{k-3}+\binom{n-5}{k-3}+\binom{n-7}{k-4}< 3\binom{n-3}{k-3}.
\end{align}
Note that for any $(i,j)\in \mathds{H}$ we have
\begin{align}\label{ineq-sijhf2}
|S_{ij}(\hf)(i)|>\left(\frac{2}{3}-\varepsilon\right)|\hf|.
\end{align}
It is equivalent to
\begin{align}\label{ineq-sijshift2}
|\hf(i,j)|+|\hf(i,\bar{j})\cup \hf(\bar{i},j)|> \left(\frac{2}{3}-\varepsilon\right)|\hf|.
\end{align}

\begin{claim}
The matching number of $\mathds{H}$ is at most one.
\end{claim}
\begin{proof}
Suppose for contradiction that $(a_1,b_1), (a_2,b_2)\in \mathds{H}$ are two disjoint edges. Let $\hg_r=\hf_{\{a_r,b_r\}}$, $r=1,2$.
Then by \eqref{ineq-sijshift2} we have
\begin{align*}
|\hg_r| = & |\hf(a_r,b_r)|+|\hf(a_r,\overline{b_r})\cup \hf(\overline{a_r},b_r)|>  \left(\frac{2}{3}-\varepsilon\right)|\hf|, r=1,2.
\end{align*}
By applying Lemma \ref{lem-3.5} with $M=\binom{n-3}{k-3}+\binom{n-4}{k-3}+\binom{n-5}{k-3}+\binom{n-7}{k-4}>2\binom{n-5}{k-3}$, we infer from \eqref{ineq-FQR-general2}:
\begin{align*}
|\hg_1\cap \hg_2|&\leq 3M+\binom{n-7}{k-5}+\binom{n-8}{k-5}< 9\binom{n-3}{k-3} \leq \frac{1}{4}|\hf|.
\end{align*}
Then by $\varepsilon\leq \frac{1}{24}$
\[
|\hg_1\cup\hg_2| > \frac{4}{3}|\hf|-2\varepsilon |\hf|-|\hg_1\cap \hg_2|> \left(\frac{13}{12}-2\varepsilon\right)|\hf|\geq |\hf|,
\]
a contradiction.
\end{proof}

Now we distinguish two cases.

\vspace{5pt}
{\noindent\bf Case 1.} $\mathds{H}=\emptyset$.
\vspace{5pt}

Then $\hf$ is initial. Using $\varrho(\hf)\leq \frac{2}{3}-\varepsilon$ we infer that
\begin{align}\label{ineq-3.5}
|\hf(\bar{1})| &\geq \left(\frac{1}{3}+\varepsilon\right)|\hf|\geq \left(\frac{1}{3}+\varepsilon\right)36\binom{n-3}{k-3}> 12\binom{n-3}{k-3}.
\end{align}
Since $\hf$ is initial, $\hf(\{r\},[3])$ is 2-intersecting on $[4,n]$ for $r=2$ and $r=3$. By applying Theorem \ref{thm-tintersecting} with $\ha=\hb$, we get for $r=2,3$
\begin{align*}
|\hf(\{r\},[3])|&\leq \binom{n-3}{k-3}.
\end{align*}
By initiality again, $\hf(\overline{[3]})$ is  4-intersecting on $[4,n]$. Applying Theorem \ref{thm-tintersecting}, we get
\begin{align*}
|\hf(\overline{[3]})|&\leq \binom{n-3}{k-4}<\binom{n-3}{k-3}.
\end{align*}
By \eqref{ineq-hfpupbd2}, we also have
\begin{align*}
|\hf(2,3)|&\leq 3\binom{n-3}{k-3}.
\end{align*}
Then
\begin{align*}
 |\hf(\bar{1})|&\leq |\hf(2,3)|+ |\hf(2,[3])|+|\hf(3,[3])|+|\hf(\overline{[3]})|< 6 \binom{n-3}{k-3},
\end{align*}
contradicting \eqref{ineq-3.5}.

\vspace{5pt}
{\noindent\bf Case 2.} $\mathds{H}$ has matching number one.
\vspace{5pt}

For notational convenience  assume that $(n-1,n)\in \mathds{H}$. Then $S_{ij}(\hf)=\hf$ for all $1\leq i<j\leq n-2$.
Define
\begin{align*}
\hg&=\{F\setminus\{n-1,n\}\colon F\in \hf,|F\cap\{n-1,n\}|=1\}\subset\binom{[n-2]}{k-1},\\[5pt]
\hh&=\{F\in \hf\colon F\subset [n-2]\}\subset \binom{[n-2]}{k}.
\end{align*}
Note that $\hg,\hh$ are both initial. Clearly, by \eqref{ineq-sijshift2} we have
\begin{align}\label{ineq-3.6}
|\hg|>\left(\frac{2}{3}-\varepsilon\right)|\hf|-|\hf(n-1,n)|>\left(\frac{2}{3}
-\varepsilon\right)|\hf|-3\binom{n-3}{k-3}.
\end{align}
Using \eqref{ineq-hfpupbd2} we infer that
\begin{align}\label{ineq-3.7}
|\hg(1)|\leq |\hf(1,n-1)|+|\hf(1,n)|< 6\binom{n-3}{k-3}.
\end{align}
By initiality we know $|\hg(2)|\leq |\hg(1)|$. Hence by \eqref{ineq-3.6} and \eqref{ineq-3.7} we obtain
\[
|\hg(\bar{1},\bar{2})|>|\hg|-|\hg(1)|-|\hg(2)|\geq |\hg|-2|\hg(1)|> \left(\frac{2}{3}
-\varepsilon\right)|\hf|-15\binom{n-3}{k-3}.
\]
Since $\varepsilon \leq\frac{1}{24}$,
\[
|\hg(\bar{1},\bar{2})|>\left(\frac{2}{3}
-\frac{1}{24}\right)|\hf|-15\binom{n-3}{k-3}
>7.5\binom{n-3}{k-3}>\binom{n-4}{k-3}.
\]
By initiality, $\hg(\bar{1},\bar{2})$ and $\hh(1,\bar{2})$ are cross 2-intersecting.  Using  \eqref{ineq-tinterha} we have
\[
|\hh(1,\bar{2})|\leq \binom{n-4}{k-4}< \binom{n-3}{k-3}.
\]
Since $\hh(\bar{1},\bar{2})$ is 3-intersecting, by \eqref{ineq-tinterhb}
\[
|\hh(\bar{1},\bar{2})|\leq \binom{n-4}{k-3}< \binom{n-3}{k-3}.
\]
Note that $\hh(\bar{1},2)\subset \hh(1,\bar{2})$ and $|\hh(1,2)|\leq 3\binom{n-3}{k-3}$. Therefore,
\begin{align*}
|\hh|&=|\hh(\bar{1},\bar{2})|+|\hh(\bar{1},2)|+| \hh(1,\bar{2})|+|\hh(1,2)|< 6 \binom{n-3}{k-3}.
\end{align*}
Then $\varrho(\hf)\leq \frac{2}{3}-\varepsilon$ implies
\[
|\hf(\overline{n-1},n)|\geq |\hf(\overline{n-1})|-|\hh|>\left( \frac{1}{3}+\varepsilon\right)|\hf|-6\binom{n-3}{k-3}> 6\binom{n-3}{k-3}.
\]
and
\[
|\hf(n-1,\overline{n})|\geq |\hf(\bar{n})|-|\hh|>\left( \frac{1}{3}+\varepsilon\right)|\hf|-6\binom{n-3}{k-3}> 6\binom{n-3}{k-3}.
\]
Now
\begin{align*}
|\hf(\overline{\{1,n-1\}},n)|\geq |\hf(\overline{n-1},n)|-|\hf(1,n)|&> 3\binom{n-3}{k-3}
\end{align*}
and
\begin{align*}
|\hf(\overline{\{1,n\}},n-1)|\geq |\hf(n-1,\overline{n})|-|\hf(1,n-1)|&> 3\binom{n-3}{k-3},
\end{align*}
which contradicts the fact that $\hf(\overline{\{1,n-1\}},n), \hf(\overline{\{1,n\}},n-1)\subset \binom{[2,n-2]}{k-1}$ are cross 2-intersecting.
\end{proof}

\section{A sharpening of the Frankl-Tokushige inequality and small diversities}

In this section, we  prove a sharpening of the Frankl-Tokushige inequality \cite{FT92}. As a corollary,  we give a short proof of the best possible upper bound on intersecting families $\hf$ with $2\leq \gamma(\hf)\leq n-k$, which was proved by Kupavskii \cite{Ku2}.

Let us recall the Frankl-Tokushige inequality \cite{FT92} as follows.

\begin{thm}[\cite{FT92}]\label{thm-ft92}
Let $\ha \subset\binom{X}{a}$ and $\hb \subset \binom{X}{b}$ be non-empty cross-intersecting families with $n=|X|\geq a+b$, $a\leq b$. Then
\begin{align}\label{ft92}
|\ha|+|\hb|\leq \binom{n}{b} - \binom{n-a}{b}+1.
\end{align}
\end{thm}

\begin{lem}
Let $a,b$ and $n$ be integers. Suppose that $n\geq a+b$ and $2\leq a\leq b$. Then
\begin{align}\label{ineq-binom}
\binom{n-1}{a-1}+\binom{n-1}{b-1}\leq \binom{n}{b}-\binom{n-a+1}{b}+ n-a+1.
\end{align}
\end{lem}

\begin{proof}
Note that \eqref{ineq-binom} is equivalent to
\[
\binom{n-1}{n-a}-\binom{n-1}{b}\leq \binom{n-a+1}{n-a} -\binom{n-a+1}{b}.
\]
It suffices to show that for $n-a+2\leq x\leq n-1$
\[
\binom{x}{n-a}-\binom{x}{b}\leq \binom{x-1}{n-a} -\binom{x-1}{b}.
\]
This is equivalent to
\[
\binom{x-1}{n-a-1}\leq \binom{x-1}{b-1}.
\]
Note that $n-a-1\geq b-1$ and $n-a-1+b-1\geq x-1$. By  the symmetry and the strict unimodality of binomial coefficients, we conclude the proof.
\end{proof}

The following result is a sharpening of \eqref{ft92}.

\begin{thm}\label{thm-6.3}
Let $n,a,b,r$ be positive integers with $n\geq a+b$. Let $\ha \subset\binom{[n]}{a}$ and $\hb \subset \binom{[n]}{b}$ be cross-intersecting families. Suppose that (i),(ii) or (iii) holds:
\begin{itemize}
  \item[(i)] $a< b$ and $|\ha|\geq r$;
  \item[(ii)] $a=b$ and $|\ha|,|\hb|\geq r$;
  \item[(iii)]  $a> b$, $|\ha|\geq r$ and $|\hb|\geq \binom{n}{b}-\binom{n-a+b}{b} +r$,
\end{itemize}
 Then for $1\leq r\leq b-1$,
\begin{align}\label{gft92-1}
|\ha|+|\hb|\leq \binom{n}{b} - \binom{n-a+1}{b}+\binom{n-a-r+1}{b-r}+r.
\end{align}
For $b\leq r\leq n-a+1$,
\begin{align}\label{gft92-2}
|\ha|+|\hb|\leq \binom{n}{b} - \binom{n-a+1}{b}+n-a+1.
\end{align}
\end{thm}
\begin{proof}
First we prove the theorem under condition (i). By Hilton's Lemma, we may assume that $\ha=\hl(n,a,|\ha|)$ and $\hb=\hl(n,b,|\hb|)$.  Without loss of generality, we further assume that $\ha,\hb$ form a saturated pair. Let us prove the theorem by induction on $a$. If $a=1$, let $\ha=\{1,2,\ldots,x\}$ with $x\geq r$. If $x> b$, then the cross-intersecting property implies $\hb=\emptyset$ and
\[
|\ha|+|\hb| =x\leq  n.
\]
If $1\leq x\leq b$ then $[x]\subset B$ for each $B\in \hb$, implying that
\[
|\ha|+|\hb| \leq \binom{n-x}{b-x}+x.
\]
Clearly for $b\leq r\leq n-a+1$ \eqref{gft92-2} follows.
Let
\[
h(x) =\binom{n-x}{b-x}+x.
\]
Note that for $x\leq b-1$
\[
h(x+1)-h(x) = 1- \binom{n-x-1}{b-x}\leq 0.
\]
It follows that $h(x)$ is decreasing on $[1,b]$.
Thus for  $1\leq r\leq b-1$,
\begin{align*}
|\ha|+|\hb|&\leq \max\left\{\binom{n-r}{b-r}+r,n\right\}= \binom{n-r}{b-r}+r
\end{align*}
and \eqref{gft92-1} follows for $a=1$.

For general $a>1$, if $|\ha|\leq\binom{n-1}{a-1}$ then by Hilton's Lemma $1\in A$ for all $A\in \ha$. By saturatedness $\{B\in \binom{[n]}{b}\colon 1\in B\}\subset \hb$. Then $\ha(1)\subset\binom{[2,n]}{a-1},\hb(\bar{1})\subset\binom{[2,n]}{b}$ are cross-intersecting and $|\ha(1)|=|\ha|\geq r$.  By the induction hypothesis,
\begin{align*}
|\ha(1)|+|\hb(\bar{1})|&\leq \binom{n-1}{b} - \binom{n-a+1}{b}+\binom{n-a-r+1}{b-r}+r,\ 1\leq r\leq b-1,\\[5pt]
|\ha(1)|+|\hb(\bar{1})|&\leq \binom{n-1}{b} - \binom{n-a+1}{b}+n-a+1,\ b\leq r\leq n-a+1.
\end{align*}
Note that $\hb(1)$ is full, that is, $|\hb(1)|=\binom{n-1}{b-1}$. Thus \eqref{gft92-1} and \eqref{gft92-2} follow.

If $|\ha|\geq \binom{n-1}{a-1}$ by Hilton's Lemma $\{A\in \binom{[n]}{a}\colon 1\in A\}\subset \ha$. It follows that $1\in B$ for all $B\in \hb$. Since $\ha(\bar{1}),\hb(1)$ are cross-intersecting and $a\leq  b-1$, we infer
\begin{align*}
|\ha|+|\hb|=|\ha(1)|+(|\ha(\bar{1})|+|\hb(1)|)&\leq \binom{n-1}{a-1}+\binom{n-1}{b-1}.
\end{align*}
By \eqref{ineq-binom} we obtain that
\begin{align*}
|\ha|+|\hb|\leq \binom{n-1}{a-1}+\binom{n-1}{b-1}&\leq   \binom{n}{b} - \binom{n-a+1}{b}+n-a+1\\[5pt]
&\leq \binom{n}{b} - \binom{n-a+1}{b}+\binom{n-a-r+1}{b-r}+r.
\end{align*}
This proves the theorem  under condition (i).

Next assume that condition (ii) holds and let $a=b=k$. Since $\ha,\hb$ are cross-intersecting, by symmetry we may assume that $|\ha|\leq \binom{n-1}{k-1}$. Then
\[
\ha\subset \left\{F\in \binom{[n]}{k}\colon 1\in F\right\}\subset\hb.
 \]
 It follows that $\ha(1),\hb(\bar{1})$ are cross-intersecting and $|\ha(1)|=|\ha|\geq r$. Since the theorem holds under condition (i) with $a=k-1$ and $b=k$, for $1\leq r\leq k-1$ we infer
\begin{align*}
|\ha(1)|+|\hb(\bar{1})|\leq \binom{n-1}{k} - \binom{n-k+1}{k}+\binom{n-k-r+1}{k-r}+r.
\end{align*}
For $k\leq r\leq n-k+1$,
\begin{align*}
|\ha(1)|+|\hb(\bar{1})|\leq \binom{n-1}{k} - \binom{n-k+1}{k}+n-k+1.
\end{align*}
Moreover, $|\hb(1)|=\binom{n-1}{k-1}$.
Thus \eqref{gft92-1} and  \eqref{gft92-2} follow under condition (ii).

Finally, we prove the theorem under condition (iii). Let $t=a-b$. Since $|\hb|\geq \binom{n}{b}-\binom{n-t}{b} +r$, we see that
\[
\left\{B\in \binom{[n]}{b}\colon B\cap[t]\neq \emptyset \right\}\subset \hb
\]
and $[t]\subset A$ for any $A\in \ha$. Then $\ha([t]), \hb(\overline{[t]})\subset \binom{[t+1,n]}{b}$ are cross-intersecting and $|\ha([t])|\geq r$, $|\hb(\overline{[t]})|\geq r$.  Since the theorem holds under condition (ii), we conclude that
for $1\leq r\leq b-1$,
\begin{align*}
|\ha|+|\hb|&=\binom{n}{b}-\binom{n-t}{b}+|\ha([t])|+|\hb(\overline{[t]})|\\[5pt]
&\leq \binom{n}{b}-\binom{n-t}{b}+\binom{n-t}{b} - \binom{n-t-b+1}{b}+\binom{n-t-b-r+1}{b-r}+r\\[5pt]
&= \binom{n}{b} - \binom{n-a+1}{b}+\binom{n-a-r+1}{b-r}+r.
\end{align*}
For $b\leq r\leq n-a+1$,
\begin{align*}
|\ha|+|\hb|&=\binom{n}{b}-\binom{n-t}{b}+|\ha([t])|+|\hb(\overline{[t]})|\\[5pt]
&\leq \binom{n}{b}-\binom{n-t}{b}+\binom{n-t}{b} - \binom{n-t-b+1}{b}+n-t-b+1\\[5pt]
&= \binom{n}{b} - \binom{n-a+1}{b}+n-a+1.
\end{align*}

\end{proof}

The following theorem was proved by Kupavskii in \cite{Ku2}. Here we give a short proof by applying Theorems  \ref{thm-6.3} and \ref{thm-f87}.

\begin{thm}[\cite{Ku2}]
Let $\hf\subset\binom{[n]}{k}$ be intersecting, $n> 2k\geq 6$ and $\gamma(\hf)\geq r$. If $1\leq r\leq k-2$ then
\begin{align}\label{gft92-7}
|\hf|\leq \binom{n-1}{k-1} - \binom{n-k}{k-1}+\binom{n-k-r}{k-r-1}+r.
\end{align}
If $k-1\leq r\leq n-k$ then
\begin{align}\label{gft92-8}
|\hf|\leq \binom{n-1}{k-1} - \binom{n-k}{k-1}+n-k.
\end{align}
\end{thm}

\begin{proof}
Without loss of generality assume that $|\hf(1)|=\max\limits_{1\leq i\leq n} |\hf(i)|$. Clearly $\gamma(\hf)\geq r$ implies $|\hf(\bar{1})|\geq r$. If $|\hf(1)|\leq \binom{n-2}{k-2}+r-1\leq \binom{n-2}{k-2}+\binom{n-3}{k-2}\leq \binom{n-1}{k-1} - \binom{n-k}{k-1}=\Delta(\ha_{k-1}(n,k))$, then by \eqref{ineq-f87} we obtain that
\[
|\hf|\leq |\ha_{k-1}(n,k)|=\binom{n-1}{k-1} - \binom{n-k}{k-1}+n-k.
\]
Thus we may assume that $|\hf(1)|\geq \binom{n-2}{k-2}+r$.  Apply Theorem \ref{thm-6.3} to $\hf(\bar{1})$ and $\hf(1)$ with $a=k$ and $b=k-1$,
for $1\leq r\leq k-2$ we have
\begin{align*}
|\hf| =|\hf(\bar{1})|+|\hf(1)|\leq \binom{n-1}{k-1} - \binom{n-k}{k-1}+\binom{n-k-r}{k-1-r}+r.
\end{align*}
For $k-1\leq r\leq n-k$, we have
\begin{align*}
|\hf| =|\hf(\bar{1})|+|\hf(1)|\leq \binom{n-1}{k-1} - \binom{n-k}{k-1}+n-k.
\end{align*}
Thus the theorem is proven.
\end{proof}

\section{Concluding remarks}

In the present paper we used the shifting ad extremis to improve earlier bounds concerning conditions $\varrho(\hf)>\frac{2}{3}-o(1)$. Let us conclude this paper by a challenging open problem.

\begin{conj}
Suppose that $\hf\subset \binom{[n]}{k}$ is intersecting, $n>100k$, $|\hf|>\binom{n-3}{k-3}$. Then
\begin{align}\label{ineq-7.1}
\varrho(\hf)\geq \frac{3}{7}.
\end{align}
\end{conj}

Note that the following construction attains the equality in \eqref{ineq-7.1}.
Let $L_1,L_2,\ldots,L_7$ be the 7 sets of size 3  corresponding to the seven lines of the Fano plane. Without loss of generality, we may assume that
\begin{align*}
&L_1=(1,2,3),\ L_2=(1,4,5),\ L_3=(1,6,7),\ L_4=(2,4,6), \\[5pt]
&L_5=(2,5,7),\ L_6=(3,5,6),\ L_7=(3,4,7).
\end{align*}
For $n\geq 10$ and $k\geq 3$, let
\[
\hf =\left\{F\in \binom{[n]}{k}\colon F\cap [7] = L_i \mbox{ for some } i\in [7]\right\}.
\]
Then it is easy to see that $\varrho(\hf) =\frac{3}{7}$.

\end{document}